\documentclass{alggeom}

\usepackage{amssymb,amsthm,amsmath,amstext}
\usepackage{mathrsfs}  % allows nice script font for sheaves

\theoremstyle{plain}
\newtheorem{theorem}{Theorem}
\newtheorem*{conjecture*}{Conjecture}
\newtheorem{prop}[theorem]{Proposition}
\newtheorem{lemma}[theorem]{Lemma}
\newtheorem{cor}[theorem]{Corollary}

\newtheorem{thmintro}{Theorem}

\newtheorem{corintro}[thmintro]{Corollary}

\theoremstyle{definition}

\theoremstyle{remark}
\newtheorem{remark}[theorem]{Remark}

\newcommand{\rk}{{\operatorname{rk}}}
\newcommand{\sgn}{{\operatorname{sgn}}}
\newcommand{\cat}[1]{\mathsf{#1}}
\newcommand{\Db}{\cat{D}^{\mathrm{b}}}
\newcommand{\Br}{{\rm Br}}
\newcommand{\NS}{{\rm NS}}

\newcommand{\pullback}{{}^*}

\newcommand{\quadform}[1]{\langle #1 \rangle}
\newcommand{\lattice}[1]{\langle #1 \rangle}
\newcommand{\inner}[1]{b#1}
\newcommand{\dual}{^{\vee}}
\newcommand{\mor}[1]{\xrightarrow{#1}}
\newcommand{\tensor}{\otimes}
\newcommand{\isom}{\cong}
\newcommand{\wt}[1]{\widetilde{#1}}

\newcommand{\inv}{^{-1}}
\newcommand{\pf}{\mathrm{pf}}
\newcommand{\sheaf}[1]{\mathscr{#1}}
\newcommand{\ko}{\sheaf{O}}
\newcommand{\kc}{\sheaf{C}}
\newcommand{\kr}{\sheaf{R}}
\newcommand{\ZZ}{\mathbb{Z}}
\newcommand{\QQ}{\mathbb{Q}}
\newcommand{\RR}{\mathbb{R}}
\newcommand{\CC}{\mathbb{C}}
\newcommand{\FF}{\mathbb{F}}
\newcommand{\PP}{\mathbb{P}}
\newcommand{\cC}{\mathcal{C}}
\newcommand{\linedef}[1]{\emph{#1}}
\newcommand{\Proj}{\mathrm{Proj}}
\usepackage[backref=page]{hyperref}
\hypersetup{pdftitle={Cubic fourfolds containing a plane and a quintic
del Pezzo surface}} \hypersetup{pdfauthor={Asher Auel, Marcello Bernardara, Michele Bolognesi, Anthony Varilly-Alvarado}} \hypersetup{colorlinks=true,linkcolor=blue,anchorcolor=blue,citecolor=blue}

\begin{document}

\title[Cubic fourfolds containing a plane and a quintic del Pezzo surface]{Cubic fourfolds containing a plane and a quintic del Pezzo surface}

\author[A.\ Auel]{Asher Auel}
\email{asher.auel@yale.edu}
\address{Department of Mathematics\\%
Yale University\\%
10 Hillhouse Avenue\\%
New Haven, CT 06511\\%
USA}

\author[M.\ Bernardara]{Marcello Bernardara}
\email{marcello.bernardara@math.univ-toulouse.fr}
\address{Institut de Math\'ematiques de Toulouse\\%
Universit\'e Paul Sabatier\\%
118 route de Narbonne\\%
31062 Toulouse Cedex 9\\%
France}

\author[M.\ Bolognesi]{Michele Bolognesi}
\email{michele.bolognesi@univ-rennes1.fr}
\address{Institut de Recherche Math\'ematique de Rennes\\%
Universit\'e de Rennes 1\\%
263 Avenue du G\'en\'eral Leclerc, CS 74205\\%
35042 Rennes Cedex\\%
France}

\author[A.\ V\'arilly-Alvarado]{Anthony V\'arilly-Alvarado}
\email{av15@rice.edu}
\address{Department of Mathematics\\%  
Rice University MS 136\\%
6100 South Main Street\\%
Houston, TX 77005\\%
USA}



\thanks{Much of this work has been developed during visits of the
authors at the Max Planck Institut f\"ur Mathematik in Bonn,
Universit\"at Duisburg--Essen, Universit\'e Rennes 1, ETH Z\"urich,
and Rice University.  The hospitality of each institute is warmly
acknowledged.  The first author was partially supported by NSF grant
MSPRF DMS-0903039 and an NSA Young Investigator grant.  The second
author was partially supported by the SFB/TR 45 `Periods, moduli
spaces, and arithmetic of algebraic varieties'.  The fourth author was
partially supported by NSF grant DMS-1103659. The authors would like
to thank N.\ Addington, F.\ Charles, J.-L.\ Colliot-Th\'el\`ene, B.\
Hassett, M.-A.\ Knus, R.\ Laza, E.\ Macr\`\i, R.\ Parimala, P.\
Stellari, and V.\ Suresh for many helpful discussions.  Comments from
an anonymous referee, who is wholeheartedly thanked, enabled us to
improve the focus and clarity of this text.}

\begin{abstract} 
We isolate a class of smooth rational cubic fourfolds $X$ containing a
plane whose associated quadric surface bundle does not have a rational
section.  This is equivalent to the nontriviality of the Brauer class
$\beta$ of the even Clifford algebra over the K3 surface $S$ of degree
2 arising from $X$. Specifically, we show that in the moduli space of
cubic fourfolds, the intersection of divisors $\cC_8 \cap \cC_{14}$
has five irreducible components.  In the component corresponding to
the existence of a tangent conic, we prove that the general member is
both pfaffian and has $\beta$ nontrivial.  Such cubic fourfolds
provide twisted derived equivalences between K3 surfaces of degree 2
and 14, hence further corroboration of Kuznetsov's derived categorical
conjecture on the rationality of cubic fourfolds.
\end{abstract}

\maketitle

%%%%%%%%%%%%%%%%%%%%%%%%%%%%%%%%%%%%%%%%%%%%%%%%%
\section*{Introduction}
\label{sec:Introduction}
%%%%%%%%%%%%%%%%%%%%%%%%%%%%%%%%%%%%%%%%%%%%%%%%%

Let $X$ be a \linedef{cubic fourfold}, i.e., a smooth cubic
hypersurface $X \subset \PP^5$ over the complex numbers.  Determining
the rationality of $X$ is an open problem in algebraic geometry.
Some classes of rational cubic fourfolds have been described by Fano
\cite{fano}, Tregub \cite{tregub,tregub:new}, and Beauville--Donagi
\cite{beauville-donagi}.  In particular, \linedef{pfaffian} cubic
fourfolds, defined by pfaffians of skew-symmetric $6\times 6$ matrices
of linear forms in $\PP^5$, are rational; see
\cite[Prop.~5(ii)]{beauville-donagi}. A cubic fourfold
is pfaffian if and only if it contains a quintic del Pezzo surface;
see \cite[Prop.~9.2 a)]{beauville:determinantal}.

Hassett \cite{hassett:special-cubics} describes, via lattice theory,
Noether--Lefschetz divisors $\cC_d$ in the moduli space $\cC$ of cubic
fourfolds, defined by the existence of 2-cycles not equivalent to a
two dimensional linear section.  For example, $\cC_{14}$ is the
closure of the locus of pfaffian cubic fourfolds and $\cC_8$ is the
locus of cubic fourfolds containing a plane.  Certain of the divisors
$\cC_d$ consist of cubic fourfolds $X$ whose \linedef{nonspecial
cohomology} is isomorphic to a Tate twist of the primitive middle
cohomology of a polarized K3 surface $S$ of degree $d$; see
\cite[Thm.~1.0.2]{hassett:special-cubics}.
Such a K3 surface
$S$ is said to be \linedef{associated} to $X$.
A natural suspicion, supported by Hassett's work
\cite{hassett:special-cubics}, \cite{hassett:rational_cubic}, is that
any rational cubic fourfold ought to have an associated K3 surface.
For example, pfaffian cubic fourfolds 
have associated K3 surfaces of degree 14.  Hassett
\cite[Thm.~4.2]{hassett:rational_cubic} identifies countably many
divisors of $\cC_8$ consisting of cubic fourfolds containing a plane
whose \emph{Clifford invariant} is trivial, implying rationality.
Though lacking an associated K3 surface of degree 8, 
%(see \cite[Thm.~1.0.2]{hassett:special-cubics}), 
such cubic fourfolds do have associated K3 surfaces of other degrees.
The work of Hassett and Tschinkel
\cite[\S7]{hassett_tschinkel:rational_curves_homolomorphic_symplectic}
highlights the important role of effectivity of 2-cycles in such
rationality considerations.  While it is expected that the general
cubic fourfold (and the general cubic fourfold containing a plane) is
nonrational, at present not a single cubic fourfold is provably
nonrational.

%\smallskip

In this work, we study rational cubic fourfolds in $\cC_8\cap\cC_{14}$
with nontrivial Clifford invariant, hence not contained in the
divisors of $\cC_8$ described by Hassett.  Let $A(X)$ be the lattice
of algebraic 2-cycles on $X$ up to rational equivalence and $d_X$ the
discriminant of the intersection form on $A(X)$.  Our main result is a
complete description of the irreducible components of $\cC_8 \cap
\cC_{14}$.

\begin{thmintro}
\label{thm:main}
The intersection of Noether--Lefschetz divisors $\cC_8 \cap \cC_{14}$
in the moduli space of cubic fourfolds has five irreducible components
indexed by the discriminant $d_X \in \{21,29,32,36,37\}$ of a general
member $X$.  The Clifford invariant of a general cubic fourfold $X$ in
$\cC_8 \cap \cC_{14}$ is trivial if and only if $d_X$ is odd.
Finally, the pfaffian locus is dense in the $d_X=32$ component.
\end{thmintro}

In particular, the general cubic fourfold in the $d_X = 32$ component
of $\cC_8 \cap \cC_{14}$ is rational and has nontrivial Clifford
invariant; the existence of such a component substantiates an unproven assertion of Hassett~\cite[Rem.~4.3]{hassett:rational_cubic}.  We also provide a geometric description of this
component:\ its general member has a \emph{tangent conic} to the
sextic discriminant curve of the associated quadric surface bundle
(see Proposition~\ref{prop:tau0}).

%\smallskip

More recently, Kuznetsov \cite{kuznetsov:cubic_fourfolds} has
established a semiorthogonal decomposition of the bounded derived
category $ \Db(X) = \langle \cat{A}_X, \ko_X, \ko_X(1),\ko_X(2)
\rangle$.  The category $\cat{A}_X$ is essentially a noncommutative
deformation of the derived category of a K3 surface.  Based on
evidence from known cases, as well as general categorical
considerations, Kuznetsov conjectures that
a cubic fourfold is rational if and only if
there is an equivalence $\cat{A}_X \isom \Db(S)$ for some K3 surface $S$.

Shortly after the first version of this article appeared, Addington
and Thomas \cite{AT_kuz=has} announced a groundbreaking result linking
the Hodge theoretic suspicion supported by Hassett's work and the
derived categorical conjecture of Kuznetsov,
%\cite{kuznetsov:cubic_fourfolds},
at least for general cubic fourfolds.

If $X$ contains a plane, the projection from that plane makes $X$
birational to the total space of a quadric surface bundle $\wt{X} \to
\PP^2$ whose discriminant $D \subset \PP^2$ is a sextic curve.
Assuming that $D$ is smooth, the double cover $S \to \PP^2$ branched
along $D$ is a K3 surface of degree 2 and the even Clifford algebra
gives a Brauer class $\beta \in \Br(S)$, called the \linedef{Clifford
invariant} of $X$.  In~\cite[Thm.~4.3]{kuznetsov:cubic_fourfolds},
Kuznetsov establishes an equivalence $\cat{A}_X \isom \Db(S,\beta)$
with the bounded derived category of $\beta$-twisted sheaves on $S$.

By classical results in the theory of quadratic forms (see
\cite[Thm.~6.3]{knus_parimala_sridharan:rank_4}), $\beta$ is trivial
if and only if the quadric surface bundle $\wt{X} \to \PP^2$ has a
rational section (see also \cite{hassett:rational_cubic}), in which
case $X$ is rational.  On the divisors of $\cC_8$ described by
Hassett, $\beta \in \Br(S)$ is trivial and Kuznetsov's conjecture is
verified.  However, there is no example in the literature with $\beta$
nontrivial and for which there nevertheless exists a different K3
surface $S'$ and an equivalence $\Db(S,\beta) \isom \Db(S')$
explaining rationality of $X$.  One application of our results is the
existence of such examples.

\begin{corintro}
\label{thm:2}
Let $X$ be a general member of the $d_X=32$ component of $\cC_8 \cap
\cC_{14}$.  Then there exists a K3 surface $S'$ of degree 14 and a
nontrivial twisted derived equivalence $\Db(S,\beta) \isom \Db(S')$.
\end{corintro}

The existence of such cubic fourfolds is not \emph{a priori} clear.
First, while the locus of pfaffian cubic fourfolds is dense in
$\cC_{14}$, it is not true that the locus of pfaffians containing a
plane is dense in (all components of) $\cC_{8} \cap \cC_{14}$.  As for
the general study of quadric bundles over surfaces, there do exist
\emph{conic} bundles over surfaces, without rational sections, whose
total space is smooth projective rational:\ these are classified over
rational minimal surfaces \cite{shokurov} (see also
\cite{bernaboloconic}); over $\PP^2$ their discriminant curves have
degree at most five \cite{beauville:Prym}.  However, there is no
analogous classification of quadric surface bundles over surfaces with
smooth projective rational total space and without rational sections.
Indeed, there do not even exist any explicit examples in the
literature.  Our results provide such examples arising from cubic
fourfolds containing a plane.

The structure of this paper is as follows.  In \S\ref{sec:lattices},
we study Hodge theoretic and geometric conditions for the
nontriviality of the Clifford invariant (see
Proposition~\ref{prop:nontriviality} and Corollary~\ref{prop:conic}).
In \S\ref{sec:sanity_check}, we analyze the irreducible components of
$\cC_8 \cap \cC_{14}$, proving the first two statements of
Theorem~\ref{thm:main}.  Throughout, we use the work of
Looijenga~\cite{looijenga} and Laza~\cite{laza}, as adapted by
Mayanskiy~\cite{lattices_mayanskiy}, on the realizability of lattices
of algebraic cycles on a cubic fourfold.  In \S\ref{sec:hpd}, we
recall some elements of the theory of homological projective duality
and prove Corollary~\ref{thm:2}.  Finally, in \S\ref{sec:example}, we
prove the final statement of Theorem~\ref{thm:main}, that the pfaffian
locus is dense in the $d_X=32$ component of $\cC_8 \cap \cC_{14}$, by
expliciting a single point in the intersection.  For the verification,
we are aided by \texttt{Magma}~\cite{magma}, adapting some of the
computational techniques developed in \cite{hassett_varilly:K3}.

%%%%%%%%%%%%%%%%%%%%%%%%%%%%%%%%%%%%%%%%%%%%%%%%%
\section{Nontriviality criteria for Clifford invariants}
\label{sec:lattices}
%%%%%%%%%%%%%%%%%%%%%%%%%%%%%%%%%%%%%%%%%%%%%%%%%

In this section, by means of straightforward lattice-theoretic
calculations, we describe a class of cubic fourfolds containing a
plane with nontrivial Clifford invariant.

If $(H,\inner{})$ is a $\ZZ$-lattice and $A \subset H$ is a
sublattice, then the orthogonal complement $A^{\perp} = \{v\in
H\,:\,\inner{(v,A)}=0\}$ is a \linedef{saturated} sublattice (i.e.,
$A^{\perp} = A^{\perp}\tensor_{\ZZ}\QQ \cap H$) and is thus a
\linedef{primitive} sublattice (i.e., $H/A^{\perp}$ is torsion free).
Denote by $d(H,\inner{})\in\ZZ$ the \linedef{discriminant}, i.e., the
determinant of a Gram matrix.

Let $X$ be a smooth cubic fourfold over $\CC$.  The integral Hodge
conjecture holds for $X$ (by \cite[Thm.~18]{voisin:aspects}, following
\cite{murre} and \cite{zucker}) and we denote by $A(X) = H^4(X,\ZZ)
\cap H^{2,2}(X)$ the lattice of integral middle Hodge classes; it
coincides with the Chow group of codimension 2 cycles up to rational
equivalence.

Now suppose that $X$ contains a plane $P$ and let $\pi : \wt{X} \to
\PP^2$ be the quadric surface bundle defined by blowing up and
projecting away from $P$.  Let $\kc_0$ be the even Clifford algebra of
$\pi$, cf.\ \cite{kuznetsov:quadrics} or \cite[\S1.5]{ABB:fibrations}.
We call the plane $P$ \linedef{good} if $\pi$ has \linedef{simple
degeneration}, i.e., the fibers of $\pi$ have at most isolated
singularities.  This is equivalent to $X$ not containing another plane
intersecting $P$; see \cite[\S1~Lemme~2]{voisin}.  This is also
equivalent to the smoothness of the discriminant divisor $D \subset
\PP^2$ (see \cite[Prop.~1.2.5]{ABB:fibrations}), which is sextic
curve.  In this case, the \linedef{discriminant cover} $f : S \to
\PP^2$ branched along $D$ is a smooth K3 surface of degree 2, and that
$\kc_0$ defines an Azumaya quaternion algebra over $S$, cf.\
\cite[Prop.~3.13]{kuznetsov:quadrics}.  We refer to the Brauer class
$\beta \in \Br(S)[2]$ of $\kc_0$ as the \linedef{Clifford invariant}
of $X$.

Let $h \in H^2(X,\ZZ)$ be the hyperplane class defined by the
embedding $X \subset \PP^5$. The \linedef{transcendental} lattice
$T(X)$, the \linedef{nonspecial cohomology} lattice $K$, and the
\linedef{primitive cohomology} lattice $H^4(X,\ZZ)_0$ are the
orthogonal complements (with respect to the cup product polarization
$\inner{}_X$) of $A(X)$, $\lattice{h^2, P}$, and $\lattice{h^2}$
inside $H^4(X,\ZZ)$, respectively.  Thus $T(X) \subset K \subset
H^4(X,\ZZ)_0$, with $T(X)=K$ for a very general cubic fourfold
containing a plane, cf.\ the proof of \cite[\S1~Prop.~2]{voisin}.  There
are polarized Hodge structures on $T(X)$, $K$, and $H^4(X,\ZZ)_0$
by restricting from $H^4(X,\ZZ)$.

Similarly, let $S$ be a smooth integral projective surface over $\CC$
and $\NS(S) = H^2(S,\ZZ) \cap H^{1,1}(S)$ its N\'eron--Severi lattice.
Let $h_1 \in \NS(S)$ be a fixed anisotropic class.  The
\linedef{transcendental} lattice $T(S)$ and the \linedef{primitive
cohomology} $H^2(S,\ZZ)_0$ are the orthogonal complements (with
respect to the cup product polarization $\inner{}_S$) of $\NS(S)$ and
$\lattice{h_1}$ inside $H^2(S,\ZZ)$, respectively.  If $f : S \to
\PP^2$ is a double cover, then we take $h_1$ to be the class of $f^*
\ko_{\PP^2}(1)$.

Let $F(X)$ be the Fano variety of lines in $X$ and $W \subset F(X)$
the divisor consisting of lines meeting $P$.  Then $W$ is identified
with the relative Hilbert scheme of lines in the fibers of $\pi$.  Its
Stein factorization $W \mor{p} S \mor{f} \PP^2$ displays $W$ as a
smooth conic bundle over the discriminant cover. The Abel--Jacobi map
$$
\Phi : H^4(X,\ZZ) \to H^2(W,\ZZ) 
$$ 
is an isomorphism of $\QQ$-Hodge structures $\Phi : H^4(X,\QQ) \to
H^2(W,\QQ)(-1)$; see \cite[\S1~Prop.~1]{voisin}. Finally, there is an
injective (see \cite[Lemma~7.28]{voisin:Hodge_I}) morphism $p^* :
H^2(S,\ZZ) \to H^2(W,\ZZ)$ of polarized Hodge structures.  Voisin
\cite[\S1~Prop.~2]{voisin} proves that $\Phi (K) \subset p\pullback
H^2(S,\ZZ)_0(-1)$ is a polarized Hodge substructure of index 2.  Here,
the Tate twist $(-1)$ increases the weight by 2 and changes the sign
of the bilinear form.  We have the following amplification.

\begin{prop}
\label{claim}
Let $X$ be a smooth cubic fourfold containing a good plane.  Then
we have that $\Phi(T(X)) \subset p\pullback T(S)(-1)$ is a polarized Hodge
substructure of index $\epsilon$ dividing 2.  In particular, $\rk\,
A(X) = \rk\, \NS(S) + 1$ and $d(A(X)) = 2^{2(\epsilon-1)} d(\NS(S))$.
\end{prop}
\begin{proof}
Using \cite[\S1~Lemme~3]{voisin} it is not hard to check that
$\Phi(T(X)) \subset p\pullback T(S)(-1)$; it remains to compute the
index of this inclusion.  Since $T(X) \subset K$ and $T(S)(-1) \subset
H^2(S,\ZZ)_0(-1)$ are saturated (hence primitive) sublattices, an
application of the snake lemma shows that
$$p\pullback T(S)(-1)/\Phi(T(X)) \subset p\pullback
H^2(S,\ZZ)_0/\Phi(K) \isom \ZZ/2\ZZ,$$
hence the index of $\Phi(T(X))$ in $p\pullback T(S)(-1)$ divides 2.

We now verify the final claims. We have $\rk\, K = \rk\, H^2(X,\ZZ) -
2 = \rk\, T(X) + \rk\, A(X) - 2$ and $\rk\, H^2(S,\ZZ)_0 = \rk\,
H^2(S,\ZZ) - 1 = \rk\, T(S) + \rk\, \NS(S) - 1$ (since $P$, $h^2$, and
$h_1$ are anisotropic vectors, respectively), while $\rk\, K = \rk\,
H^2(S,\ZZ)_0$ and $\rk\, T(X) = \rk\, T(S)$ by
\cite[\S1~Prop.~2]{voisin} and the above, respectively.  The
discriminant calculation follows from standard lattice theory.
\end{proof}

If $S$ has Picard rank 1, then $X$ has no associated K3 surface in the
sense of Hassett.  Thus, if we are looking for rational cubic
fourfolds, we must consider $S$ with Picard rank at least 2.  Let $Q
\in A(X)$ be the class of a fiber of $\pi : \wt{X} \to \PP^2$.  Then
$P + Q = h^2$, see \cite[\S1]{voisin}.

\begin{prop}
\label{prop:nontriviality}
Let $X$ be a smooth cubic fourfold containing a good plane $P$.  If
$A(X)$ has rank 3 and even discriminant (e.g., if the K3 surface $S$
of degree 2 has Picard rank 2 and even N\'eron--Severi discriminant)
then the Clifford invariant $\beta \in \Br(S)$ of $X$ is nontrivial.
\end{prop}

\begin{proof}
The Clifford invariant $\beta \in \Br(S)$ of the quadric
surface bundle $\pi : \wt{X} \to \PP^2$ is trivial if and only if
$\pi$ has a rational section; see
\cite[Thm.~6.3]{knus_parimala_sridharan:rank_4} or
\cite[2~Thm.~14.1,~Lemma~14.2]{scharlau:book}.  Such a section exists
if and only if there exists an algebraic cycle $R \in A(X)$ such that
$R.Q=1$; see \cite[Thm.~3.1]{hassett:rational_cubic} or
\cite[Prop.~4.7]{kuznetsov:cubic_fourfolds}.

Suppose that such a cycle $R$ exists and consider the sublattice
$\lattice{h^2, Q, R} \subset A(X)$. It is straightforward to see that
its intersection form has a Gram matrix whose determinant is congruent
to 5 modulo 8 for any possible choice of $R$ (cf.\
\cite[Lemma~4.4]{hassett:rational_cubic}), so this lattice cannot be a
finite index sublattice of $A(X)$, which has even discriminant by
hypothesis.  Hence no such 2-cycle $R$ exists and thus $\beta$ is
nontrivial. The claim follows from Proposition~\ref{claim}.
\end{proof}

We now provide an explicit geometric condition for the nontriviality
of the Clifford invariant, which will be necessary in
\S\ref{sec:example}.  We say that a cubic fourfold $X$ containing a
plane has a \linedef{tangent conic} if there exists a conic $C \subset
\PP^2$ everywhere tangent to the discriminant curve $D\subset \PP^2$
of the associated quadric surface bundle.

\begin{cor}
\label{prop:conic}
Let $X$ be a smooth cubic fourfold containing a good plane.  If $X$
has a tangent conic and the K3 surface $S$ has Picard rank 2 then the
Clifford invariant $\beta \in \Br(S)$ of $X$ is nontrivial.
\end{cor}
\begin{proof}
Consider the pull back of the cycle class of $C$ to $S$ via the
discriminant double cover $f : S\to \PP^2$.  Then $f^* C$ has two
components $C_1$ and $C_2$.  The sublattice of $\NS(S)$ generated by
$h_1 = f\pullback\ko_{\PP^2}(1) = (C_1 + C_2)/2$ and $C_1$ has
discriminant $-8$.  As $S$ has Picard rank 2, the classes $h_1$ and
$C_1$ generate $\NS(S)$ (see \cite[\S2]{elsenhal-janhel:k3} for
further details). Now apply Proposition~\ref{prop:nontriviality}.
\end{proof}

\section{The Clifford invariant on $\cC_8 \cap\cC_{14}$}
\label{sec:sanity_check}

In this section, we first prove that $\cC_8 \cap \cC_{14}$ has five
irreducible components and we describe each of them in lattice
theoretic terms.  We then completely analyze the (non)triviality of
the Clifford invariant of the general cubic fourfold (i.e., such that
$A(X)$ has rank 3) in each irreducible component.  One of the
components corresponds to cubic fourfolds containing a plane and
having a tangent conic (i.e., those considered in
Corollary~\ref{prop:conic}), where we already know the nontriviality
of the Clifford invariant.  Another component corresponds to cubic
fourfolds containing two disjoint planes, where we already know the
triviality of the Clifford invariant.  There are another two
components of $\cC_8 \cap \cC_{14}$ whose general elements have
trivial Clifford invariant (see Proposition~\ref{oddity}).  
\medskip

A cubic fourfold $X$ is in $\cC_8$ or $\cC_{14}$ if and only if $A(X)$
has a primitive sublattice $K_8=\lattice{h^2,P}$ or
$K_{14}=\lattice{h^2,T}$.  This follows from the definition of
$\cC_d$, and because for any $d \not\equiv 0 \bmod 9$ there is a
unique lattice (up to isomorphism) of rank 2 that represents 3 and has
discriminant $d$.

Thus a cubic fourfold $X$ in $\cC_8 \cap \cC_{14}$ has
a sublattice $\lattice{h^2, P, T} \subset A(X)$ with Gram
matrix
\begin{equation}
\label{eq:hypothetical}
\begin{array}{cccc}
 & h^2 & P & T\\
h^2 & 3 & 1 & 4\\
P & 1 & 3 & \tau \\
T & 4 & \tau & 10
\end{array}
\end{equation}
for some $\tau \in \ZZ$ depending on $X$.  There may be \emph{a
priori} restrictions on the possible values of $\tau$.  

Denote by $A_\tau$ the lattice of rank 3 whose bilinear form has Gram
matrix \eqref{eq:hypothetical}.  We will write $\cC_\tau =
\cC_{A_\tau} \subset \cC$ for the locus of smooth cubic fourfolds such
that there is a primitive embedding $A_\tau \subset A(X)$ of lattices
preserving $h^2$.  If nonempty, each $\cC_\tau$ is a subvariety of
codimension 2 by a variant of the proof of
\cite[Thm.~3.1.2]{hassett:special-cubics}.

We will use the work of Laza~\cite{laza}, Looijenga~\cite{looijenga},
and Mayanskiy~\cite[Thm.~6.1]{lattices_mayanskiy} to classify exactly
which values of $\tau$ are supported by cubic fourfolds. This proves
the first assertion in Theorem~\ref{thm:main}.  

\begin{theorem}
\label{theorem:comps}
The irreducible components of $\cC_8 \cap \cC_{14}$ are the
subvarieties $\cC_\tau$ for $\tau \in \{-1,0,1,2,3\}$.  Moreover, the
general cubic fourfold $X$ in $\cC_\tau$ satisfies $A(X)\isom A_\tau$.
\end{theorem}
\begin{proof}
By construction, $\cC_8 \cap \cC_{14}$ is the union of $\cC_\tau$ for
all $\tau \in \ZZ$.
First we determine the values of $\tau$ for which the component
$\cC_\tau$ is possibly
nonempty.  If $X$ is a smooth cubic fourfold, then $A(X)$ is positive
definite by the Riemann bilinear relations.  Hence, to be realized as
a sublattice of some $A(X)$, the lattice $A_\tau$ must be positive
definite, which by Sylvester's criterion, is equivalent to $A_\tau$
having positive discriminant.  As $d(A_\tau)=-3\tau^2+8\tau+32$, the
only values of $\tau$ making a positive discriminant are
$-2$, $-1$, $0$, $1$, $2$, $3$, and $4$.

Then, we prove that $\cC_\tau$ is empty for $\tau=-2,4$ by
demonstrating \emph{roots} (i.e., primitive vectors of norm 2) in
$A_{\tau,0} = \lattice{h^2}^{\perp}$ (see \cite[\S4~Prop.~1]{voisin},
\cite[\S2]{looijenga}, or \cite[Def.~2.16]{laza} for details on
roots).  Indeed, the vectors $(1, -3, 0)$ and $(0, -4, 1)$ form a
basis for $A_{\tau,0} \subset A_\tau$; for $\tau=-2$, we find short
roots $(-2,2,1)$ and $(2,-10,1)$; for $\tau=4$, we find short roots
$\pm(1,1,-1)$.  Hence $\cC_\tau$ is possibly nonempty only for $\tau
\in \{-1,0,1,2,3\}$.  The corresponding discriminants $d(A_\tau)$ are
 $21$, $32$, $37$, $36$, and $29$.

For the remaining values of $\tau$, we prove that $\cC_\tau$ is
nonempty.  To this end, we verify conditions 1)--6) of
\cite[Thm.~6.1]{lattices_mayanskiy}, proving that $A_\tau = A(X)$ for
some cubic fourfold $X$.   Condition 1) is true by definition.  For
condition 2),
% the vectors $(1, -3, 0)$ and
% $(0, -4, 1)$ form a basis for $A_0=\lattice{h^2}^{\perp}$.  Then
letting $v = (x,-3x-4y,y) \in A_{\tau,0}$ we
see that
\begin{equation}
\label{eq:roots}
\inner{(v,v)}= 2\bigl(12 x^2 + (36 - 3 \tau) x y + (29 - 4 \tau) y^2\bigr)
\end{equation}
is even.  For condition 5), letting $w = (x,y,z) \in A_\tau$, we
compute that
\begin{equation}
\label{eq:alt_form}
\inner{(h^2,w)}^2 - \inner{(w,w)} = 
2\bigl( 3 x^2 - y^2 + z^2 + 2 x y + 8 x z + (4-\tau) y z \bigr)
\end{equation}
is even.  For conditions 3)--4), given each of the five values of
$\tau$, we use standard Diophantine techniques to prove the
nonexistence of short and long roots of \eqref{eq:roots}.

Finally, for condition 6), let $q_{K_\tau} : A^*_\tau/A_\tau \to
\QQ/2\ZZ$ be the discriminant form of \eqref{eq:alt_form}, restricted
to the discriminant group $A^*_\tau/A_\tau$ of the lattice $A_\tau$.
Appealing to Nikulin~\cite[Cor.~1.10.2]{nikulin}, it suffices to check
that the signature satisfies $\sgn(q_{K_\tau}) \equiv 0 \bmod 8$; cf.\
\cite[Rem.~6.3]{lattices_mayanskiy}.  Employing the notation of
\cite[Prop.~1.8.1]{nikulin}, we compute the finite quadratic form
$q_{K_\tau}$ in each case:
\begin{equation}
\label{table}
\renewcommand{\arraystretch}{1.2}
\begin{array}{|c||c|c|c|c|c|}
\hline
\tau  & -1 & 0  & 1  & 2  & 3   \\\hline
d(A_\tau)  & 21 & 32 & 37 & 36 & 29  \\\hline
A^*_\tau/A_\tau & \ZZ/3\ZZ \times \ZZ/7\ZZ & \ZZ/2\ZZ\times\ZZ/16\ZZ &
\ZZ/37\ZZ & \ZZ/2\ZZ\times\ZZ/2\ZZ\times\ZZ/9\ZZ & \ZZ/29\ZZ  \\\hline
q_{K_\tau} & q_1^3(3)\oplus q_1^7(7) & q_3^2(2)\oplus q_1^2(2^4) & q_\theta^{37}(37) & q_3^2(2)\oplus
q_1^2(2)\oplus q_1^3(3^2)  & q_\theta^{29}(29)\\\hline
\end{array}
\end{equation}
where $\theta$ represents a nonsquare class modulo the respective odd
prime.  In each case of \eqref{table}, we verify the signature
condition using the formulas in \cite[Prop.~1.11.2]{nikulin}.  
% In total, we've proved a stronger statement, that for each $\tau \in
% \{-1,0,1,2,3\}$ the lattice $A_\tau$ arises as the full lattice
% $A(X)$ for a smooth cubic fourfold $X$.
% \begin{enumerate}
% \item $q(h^2,h^2) = 3$.
%
% \item
% The vectors $(1, -3, 0)$ and $(0, -4, 1)$ form a basis for
% $A_0=\lattice{h^2}^{\perp}$ and $q|_{A_0}$ is even.
%
% \item For $\delta = (a, -3a -4b, b) \in A_0$, we have
% $$
% q(\delta,\delta)=2(12a^2+33ab+25b^2),
% $$ 
% which does not take the values 2 or 6.
%
% \item For $\delta =  (x,y,z) \in A$, we have
% \begin{align*}
% q(h^2,\delta)^2 &= 9x^2+y^2+16z^2 + \textrm{terms with even coefficients}, \\
% q(\delta, \delta) &= 3x^2+3y^2+10z^2 + \textrm{terms with even coefficients},
% \end{align*} 
% so that $q(h^2,\delta)^2 - q(\delta, \delta)$ takes even values for
% all $\delta \in A$. 
% \item The discriminant group $A^*/A$ is $\ZZ/37\ZZ$, thus $l(A^*/A)=1$
% in the terminology of \cite{lattices_mayanskiy}.  We then appeal to
% \cite[Rem.~6.3]{lattices_mayanskiy}.
% \end{enumerate}
%
% We conclude that the lattice $(A,\inner{})$ arises as $A(X)$ for some
% smooth cubic fourfold $X$ for each $\tau$ as claimed.  
%
% As $(A,q)$ contains saturated rank 2
% sublattices containing $h^2$ with determinants 8 (the upper left
% $2\times 2$ minor) and 14 (the outer $2\times 2$ minor), we have that
% $X\in\cC_8 \cap \cC_{14}$.

Finally, for the five values of $\tau$, we prove that $\cC_\tau$ is
irreducible.  As the rank of $A(X)$ is an upper-semicontinuous
function on $\cC$, the general cubic fourfold $X$ in $\cC_8 \cap
\cC_{14}$ has $A(X)$ of rank 3 (by the argument above), of which
$A_\tau$ is a finite index sublattice for some $\tau$.  Each proper
finite overlattice $B$ of $A_\tau$, such that $B$ (along with its
sublattices $K_8$ and $K_{14}$) is primitively embedded into
$H^4(X,\ZZ)$, will give rise to an irreducible component of $\cC_\tau$.
We will prove that no such proper finite overlattices exist.  For
$\tau \in \{21,37,29\}$, the discriminant of $A_\tau$ is squarefree,
so there are no proper finite overlattices.  In the case $\tau=0,2$,
we note that $B_0 = \lattice{h^2}^{\perp}$ is a proper finite
overlattice of the binary lattice $A_{\tau,0}$ (as $\lattice{h^2}
\subset B$ is assumed primitive).  We then directly compute that each
such $B_0$ has \emph{long roots} (i.e., vectors of norm 6 whose
pairing with any other vector is divisible by 3).
% For $\tau=0$, we have that $A_{\tau,0}$ has Gram matrix 
% $$
% \begin{array}{cc}
% 24 & 36\\
% 36 & 58
% \end{array}
% $$
% and the only overlattice of index 2 is the diagonal lattice
% $\text{diag}(4,6)$ and of index 4 is the diagonal lattice $\text{diag}(1,6)$.
% For $\tau=2$, we have that $A_{\tau,0}$ has Gram matrix 
% $$
% \begin{array}{cc}
% 24 & 30\\
% 30 & 42
% \end{array}
% $$
% and the only overlattice of index 2 has Gram matrix 
% $$
% \begin{array}{cc}
% 6 & 3\\
% 3 & 6
% \end{array}
% $$  
% and of index 3 is the diagonal lattice $\text{diag}(2,6)$.
Therefore, no such proper finite overlattices exist.
\end{proof}

We now address the question of the (non)triviality of the Clifford
invariant.  This proves the second assertion of Theorem~\ref{thm:main}.

\begin{prop}\label{oddity}
Let $X$ be a general cubic fourfold in $\cC_{8}\cap \cC_{14}$ (so that
$A(X)$ has rank 3).  The Clifford invariant $\beta \in \Br(S)$ of $X$ is trivial if and only if
$\tau$ is odd.
\end{prop}
\begin{proof}
If $\tau$ is odd then, as in the proof of Proposition~\ref{prop:tau0},
$(P+T).Q = -\tau$ is odd, hence the Clifford invariant $\beta \in
\Br(S)$ is trivial by an application of the criteria in
\cite[Thm.~3.1]{hassett:rational_cubic} or
\cite[Prop.~4.7]{kuznetsov:cubic_fourfolds} (cf.\ the proof of
Proposition~\ref{prop:nontriviality}).  If $\tau$ is even, then
$A_d=A(X)$ has rank 3 and even discriminant, hence $\beta$ is
nontrivial by Proposition~\ref{prop:nontriviality}. 
\end{proof}

For $\tau=-1$, the component $\cC_{\tau}$ consists of cubic fourfolds
containing two disjoint planes; see
\cite[4.1.3]{hassett:special-cubics}.  We now give a geometric
description of the general member of the component $\cC_\tau$ for
$\tau=0$ (i.e., where $d_X=32$).

\begin{prop}
\label{prop:tau0}
Let $X$ be a smooth cubic fourfold containing a good plane $P$ and
having a tangent conic such that $A(X)$ has rank 3.  Then $X$ is in
the component $\cC_{\tau}$ for $\tau=0$.
\end{prop}
\begin{proof}
Since $X$ has a tangent conic and $A(X)$ has rank 3, $A(X)$ has
discriminant $8$ or $32$ and $X$ has nontrivial Clifford invariant by
Proposition~\ref{claim} and Corollary~\ref{prop:conic}.  As the
sublattice $\lattice{h^2,P} \subset A(X)$ is primitive, we can choose
a class $T \in A(X)$ such that $\lattice{h^2,P,T} \subset A(X)$ has
discriminant 32.  Adjusting $T$ by a multiple of $P$, we can assume
that $h^2.T=4$.  Write $\tau = P.T$.

Adjusting $T$ by multiples of $h^2-3P$ keeps $h^2.T=4$ and adjusts
$\tau$ by multiples of 8.  A similar trick is employed in
\cite[Prop.~4.2]{AT_kuz=has}.  
% Transforming $T$ to $-T + 3h^2 - P$ keeps
% $h^2.T=4$ while changes the sign of $b$ modulo 8.  
The discriminant being 32, we are left with two possible choices ($\tau=0,4$) for
the Gram matrix of $\lattice{h^2,P,T}$ up to isomorphism:
$$
\begin{array}{cccc}
 & h^2 & P & T\\
h^2 & 3 & 1 & 4\\
P & 1 & 3 & 0 \\
T & 4 & 0 & 10
\end{array}
\qquad
\begin{array}{cccc}
 & h^2 & P & T\\
h^2 & 3 & 1 & 4\\
P & 1 & 3 & 4 \\
T & 4 & 4 & 12
\end{array}
$$
In these cases, we compute that $K \cap \lattice{h^2,P,T}$ (i.e., the
orthogonal complement of $\lattice{h^2,P}$ in $\lattice{h^2,P,T}$) is
generated by $3h^2-P-2T$ and $h^2+P-T$ and has discriminant $16$ and
$5$, respectively.  We calculate that $\NS(S) \cap
H^2(S,\ZZ)_0$ (i.e., the orthogonal complement of $\lattice{h_1}$ in
$\NS(S)$) is generated by $h_1-C_1$ and has discriminant $-4$ (see
Corollary~\ref{prop:conic} for definitions).  Arguing as in the proof
of Proposition~\ref{claim}, there is a lattice inclusion $\Phi(K \cap
\lattice{h^2,P,T}) \subset \NS(S) \cap H^2(S,\ZZ)_0(-1)$ having index
dividing 2, which rules out the second case above by comparing
discriminants.
\end{proof}

In Proposition~\ref{oddity}, we isolate three classes of smooth cubic
fourfolds $X \in \cC_8 \cap \cC_{14}$ with \emph{trivial} Clifford
invariant.  While the component $\cC_\tau$ for $\tau=-1$ is in the
complement of the pfaffian locus (see \cite[Prop.~1b]{tregub:new}), we
wonder if the pfaffian locus is dense in the other four
components.

%%%%%%%%%%%%%%%%%%%%%%%%%%%%%%%%%%%%%%%%%%%%%%%%%
\section{The twisted derived equivalence}
\label{sec:hpd}
%%%%%%%%%%%%%%%%%%%%%%%%%%%%%%%%%%%%%%%%%%%%%%%%%

Homological projective duality (HPD) can be used to obtain a
significant semiorthogonal decomposition of the derived category of a
pfaffian cubic fourfold.  As the universal pfaffian variety is
singular, a noncommutative resolution of singularities is required to
establish HPD in this case.  A \emph{noncommutative resolution of
singularities} of a scheme $Y$ is a coherent $\ko_Y$-algebra $\kr$
with finite homological dimension that is generically a matrix algebra
(these properties translate to ``smoothness'' and ``birational to
$Y$'' from the categorical language).  We refer to
\cite{kuznetsov:hpd_general} for details on HPD. The following is a straightforward application of HPD for Grassmannians \cite{kuznetsov:hpd_for_grassmann}.

\begin{theorem}
\label{hpd}
Let $W$ be a $\CC$-vector space of dimension 6 and $Y \subset
\PP(\wedge^2 W\dual)$ the universal pfaffian cubic hypersurface.
There exists a noncommutative resolution of singularities $(Y,\kr)$
that is HP dual to the grassmannian $\mathrm{Gr}(2,W)$.  In
particular, the bounded derived category of a smooth pfaffian cubic
fourfold $X$ admits a semiorthogonal decomposition
$$
\Db(X)= \langle\Db(S'), \ko_X, \ko_X(1), \ko_X(2)  \rangle,
$$
where $S'$ is a smooth K3 surface of degree 14.  In particular,
$\cat{A}_X \isom \Db(S')$.
\end{theorem}

Assuming Theorem~\ref{thm:main}, we can now give a proof of Corollary~\ref{thm:2}.

\begin{proof}[Proof of Corollary~\ref{thm:2}]
By Theorem~\ref{thm:main}, $X$ is a smooth pfaffian cubic fourfold
containing a good plane with nontrivial Clifford invariant $\beta \in
\Br(S)$.  Being pfaffian, $X$ is rational. Let $S'$ be the K3 surface
of degree 14 arising from Theorem~\ref{hpd} via projective duality.
Then by \cite[Thm.~4.3]{kuznetsov:cubic_fourfolds} and
Theorem~\ref{hpd}, the category $\cat{A}_X$ is equivalent to both
$\Db(S,\beta)$ and $\Db(S')$.
\end{proof}

\begin{remark}
\label{rem:no_twist}
By \cite[Rem.~7.10]{huybrechts-stellari}, given any K3 surface $S$ and
any nontrivial $\beta \in \Br(S)$, there is \emph{no} equivalence
between $\Db(S,\beta)$ and $\Db(S)$.  Thus any $X$ as in
Corollary~\ref{thm:2} validates Kuznetsov's conjecture, 
but not via the K3 surface $S$. Moreover, $S$ and $S'$ are
twisted Fourier--Mukai partners: by
\cite[Thm.~5.1]{can-stel}, the equivalence $\Db(S,\beta)\cong \Db(S')$
is a Fourier--Mukai functor whose kernel is a $\beta^{-1} \boxtimes
\ko_{S'}$-twisted complex on $S \times S'$. Hence, by \cite[Thm. 4.3]{huybrechts-stellari}, 
$S$ and $S'$ have Hodge
isogenous (twisted) transcendental lattices. 
\end{remark}

%%%%%%%%%%%%%%%%%%%%%%%%%%%%%%%%%%%%%%%%%%%%%%%%%
\section{A pfaffian containing a plane}
\label{sec:example}
%%%%%%%%%%%%%%%%%%%%%%%%%%%%%%%%%%%%%%%%%%%%%%%%%

In this section, we prove the final claim of Theorem~\ref{thm:main} by
exhibiting a smooth pfaffian cubic fourfold $X$ containing a good
plane, having a tangent conic, and such that $A(X)$ has rank 3.
Indeed, by Propositions~\ref{prop:nontriviality} and \ref{prop:tau0},
such an $X$ has nontrivial Clifford invariant and is in the $\tau=0$
(i.e., $d_X=32$) component of $\cC_8 \cap \cC_{14}$.  In particular,
the pfaffian locus nontrivially intersects, and hence in dense in
(since it is open in $\cC_{14}$), the component $\cC_\tau$ with
$\tau=0$.  % This proves Theorem~\ref{thm:2}.

\begin{theorem}
\label{main}
Let A be the $6\times 6$ antisymmetric matrix of linear forms in $\QQ[x,y,z,u,v,w]$ 
$$
\left(
\begin{array}{cccccc}
 0  & y + u & x + y + u & u             & z                 & y + u + v\\
    &  0    & x + y + z & x + z + u + w & y + z + u + v + w & x + y + z + u + v + w\\
    &       &  0        & x + y + u + w & x + y + u + v + w & x + y + z + v + w\\
    &       &           &             0 & x + u + v + w     & x + u + w\\
    &       &           &               &  0                & z + u + w\\
    &       &           &               &                   & 0
\end{array}
\right)
$$
and let $X \subset \PP^5$ be the
cubic fourfold defined by the vanishing of the pfaffian of $A$:
\begin{align*}
&(x - 4y - z)u^2 + (-x - 3y)uv + (x - 3y)uw + (x - 2y - z)vw - 2yv^2 + xw^2\\
~& + (2x^2 + xz - 4y^2 + 2z^2)u + (x^2 - xy - 3y^2 + yz - z^2)v + (2x^2 + xy + 3xz - 3y^2 + yz)w \\
~& + x^3 + x^2y + 2x^2z - xy^2 + xz^2 - y^3 + yz^2 - z^3.
\end{align*}
Then:
\begin{enumerate}
\item \label{main:a} $X$ is smooth, rational, and contains the plane $P = \{ x=y=z=0\}$.

\item \label{main:b} The discriminant divisor $D \subset \PP^2$ of the 
quadric surface bundle $\pi : \wt{X} \to \PP^2$ is the sextic curve
given by the vanishing of:
\begin{align*}
d = x^6 & + 6x^5y + 12x^5z + x^4y^2 + 22x^4yz + 28x^3y^3 - 38x^3y^2z
      + 46x^3yz^2 + 4x^3z^3 \\
     & + 24x^2y^4 - 4x^2y^3z - 37x^2y^2z^2 
      -36x^2yz^3 - 4x^2z^4 + 48xy^4z - 24xy^3z^2 \\
      & + 34xy^2z^3 + 4xyz^4 + 20y^5z + 20y^4z^2 - 8y^3z^3 - 11y^2z^4
      - 4yz^5.
\end{align*}
This curve is smooth; in particular, $\pi$ has simple degeneration and
the discriminant cover is a smooth K3 surface $S$ of  degree 2.

\item \label{main:c} The conic $C \subset \PP^2$ defined by the vanishing of $x^2+yz$
is tangent to the degeneration divisor $D$ at six points (five of
which are distinct).

\item \label{main:d} The K3 surface $S$ has (geometric) Picard rank $2$.
\end{enumerate}
In particular, the Clifford invariant of $X$ is geometrically nontrivial.
\end{theorem}

\begin{proof}
An application of the jacobian criterion shows that $X$ and $D$ are
smooth.  The inclusion $P \subset X$ is checked by inspecting the
expression for $\pf(A)$; every monomial is divisible by $x$, $y$ or
$z$.  Rationality comes from being a pfaffian cubic fourfold.
Smoothness of $D$ and $X$ implies that $\pi$ has simple degeneration;
see
% see \cite[\S3]{colliot_skorobogatov:quadriques},
\cite[Rem.~7.1]{hassett_varilly:K3} or
\cite[Prop.~1.6]{ABB:fibrations}.  This proves parts \eqref{main:a}
and \eqref{main:b}.

For part \eqref{main:c}, we write the equation for the degeneration divisor as $d =
(x^2+yz)f + g^2$, where
\begin{align*}
f = {} & x^4 + 6x^3y + 12x^3z + x^2y^2 + 21x^2yz - 25x^2z^2 +
28xy^3 \\
 & - 24xy^2z + 34xyz^2
       + 4xz^3 + 20y^4 - 5y^3z - 8y^2z^2 - 11yz^3 - 4z^4,\\
g = {} & 2xy^2 + 5y^2z - 5x^2z.
\end{align*}
Hence the conic $C \subset \PP^2$ defined by $x^2+yz$ is tangent to
$D$ along the zero-dimensional scheme of length 6 given by the
intersection of $C$ and the vanishing locus of $g$. 
% A \texttt{Magma} calculation shows that this scheme consists of 5 distinct points.

For part \eqref{main:d}, the surface $S$ is the smooth sextic in
$\PP(1,1,1,3) = \Proj\QQ[x,y,z,w]$ given by $w^2 = d(x,y,z)$, which is
the double cover of $\PP^2$ branched along $D$.  In these coordinates,
the discriminant cover $f : S \to \PP^2$ is simply the restriction to
$S$ of the projection $\PP(1,1,1,3) \dasharrow \PP^2$ away from the
hyperplane $\{w=0\}$.  Let $C \subset \PP^2$ be the conic from part
\eqref{main:c}.  As discussed in Corollary~\ref{prop:conic}, the curve
$f^*C$ consists of two $({-2})$-curves $C_1$ and $C_2$.
%  Explicitly, we have
% \begin{align*}
% C_1 \; : \; x^2 + yz &= w - (2xy^2 + 5y^2z - 5x^2z) = 0, \\
% C_2 \; : \; x^2 + yz &= w + (2xy^2 + 5y^2z - 5x^2z) = 0.
% \end{align*}
These curves generate a sublattice of
$\NS(S)$ of rank $2$. Hence $\rho(\overline{S})\geq \rho(S) \geq 2$, where $\overline{S}=S \times_\QQ \CC$.

We show next that $\rho(\overline{S}) \leq 2$. Write $S_p$ for the
reduction mod $p$ of $S$ and $\overline{S}_p = S_p \times_{\FF_p}
\overline{\FF}_p$. Let $\ell \neq 3$ be a prime and write $\phi(t)$
for the characteristic polynomial of the action of absolute Frobenius
on $H^2_{\textrm{\'et}}(\overline{S}_3,\QQ_\ell)$. Then
$\rho(\overline{S}_3)$ is bounded above by the number of roots of
$\phi(t)$ that are of the form $3\zeta$, where $\zeta$ is a root of
unity \cite[Prop.~2.3]{van_luijk}. Combining the Lefschetz trace
formula with Newton's identities and the functional equation that
$\phi(t)$ satisfies, it is possible calculate $\phi(t)$ from knowledge
of $\#S(\FF_{3^n})$ for $1 \leq n \leq 11$; see \cite{van_luijk} for
details.

Let $\widetilde{\phi}(t) = 3^{-22}\phi(3t)$, so that the number of roots of $\widetilde{\phi}(t)$ that are roots of unity gives an upper bound for $\rho(\overline{S}_3)$. Using {\tt Magma}, we compute
\[
\widetilde{\phi}(t) = 
\frac{1}{3}(t - 1)^2(3t^{20} + t^{19} + t^{17} + t^{16} + 2t^{15} + 3t^{14} + t^{12} + 3t^{11}
				  + 2t^{10} + 3t^9 + t^8 + 3t^6 + 2t^5 + t^4 + t^3 + t + 3)
\]
The roots of the degree $20$ factor of $\widetilde{\phi}(t)$ are not
integral, and hence they are not roots of unity.  We conclude that
$\rho(\overline{S}_3) \leq 2$.  By \cite{van_luijk}, we have
$\rho(\overline{S}) \leq \rho(\overline{S}_3)$, so $\rho(\overline{S})
\leq 2$. It follows that $S$ (and $\overline{S}$) has Picard rank 2.
This concludes the proof of part \eqref{main:d}.
Finally, the nontriviality of the Clifford invariant follows from
Proposition~\ref{prop:nontriviality} and Corollary~\ref{prop:conic}.
\end{proof}

A satisfying feature of Theorem~\ref{main} is that we can write out a
representative of the Clifford invariant of $X$ explicitly, as a
quaternion algebra over the function field of the K3 surface $S$.  We
first prove a handy lemma, of independent interest for its arithmetic
applications (see e.g.,
\cite{hassett_varilly:K3_hasse,hassett_varilly:K3}).

\begin{lemma}
\label{lem:handy}
Let $K$ be a field of characteristic $\neq 2$ and $q$ a nondegenerate
quadratic form of rank 4 over $K$ with discriminant extension $L/K$.
For $1 \leq r \leq 4$ denote by $m_r$ the determinant of the
leading principal $r \times r$ minor of the symmetric Gram matrix of $q$.
Then the class $\beta \in \Br(L)$ of the even Clifford algebra of $q$
is the quaternion algebra $(-m_2, -m_1 m_3)$.
\end{lemma}
\begin{proof}
On $n\times n$ matrices $M$ over $K$, symmetric gaussian elimination is the
following operation:
$$
M=
\begin{pmatrix}
a & v^t \\
v & A
\end{pmatrix}
\mapsto
\begin{pmatrix}
a & 0 \\
0 & A - a\inv vv^t
\end{pmatrix}
$$
where $a \in K^\times$, $v \in K^{n-1}$ is a column vector, and $A$ is
an $(n-1)\times(n-1)$ matrix over $K$.  Then $m_1=a$ and the element
in the first row and column of $A - a\inv vv^t$ is precisely
$m_2/m_1$.  By induction, $M$ can be diagonalized, using symmetric
gaussian elimination, to the matrix
$$
\text{diag}({ m_1, m_2/m_1, \dotsc, m_{n}/m_{n-1} }).
$$ 
For $q$ of rank 4 with symmetric Gram matrix $M$, we have 
$$
q = \quadform{m_1} \tensor \quadform{1,m_2,m_1 m_2 m_3, m_1 m_3 m_4}
$$ 
so that over $L = K(\sqrt{m_4})$, we have that $q\tensor_K L =
\quadform{m_1} \tensor \quadform{1,m_2,m_1 m_3, m_1 m_2 m_3}$, which
is similar to the norm form of the quaternion $L$-algebra with symbol
$(-m_2,-m_1 m_3)$.  Thus the even Clifford algebra of $q$ is Brauer
equivalent to $(-m_2,-m_1 m_3)$ over $L$.
\end{proof}

\begin{prop}
The Clifford invariant of the fourfold $X$ of Theorem~\ref{main} is
represented by the unramified quaternion algebra $(b,ac)$ over the
function field of the K3 surface $S$, where
$$
a = x - 4y - z, \quad
b = x^2 + 14xy - 23y^2 - 8yz,
$$
and
$$
c = 3x^3 + 2x^2y - 4x^2z + 8xyz + 3xz^2 - 16y^3 - 11y^2z - 8yz^2 - z^3.
$$
\end{prop}

\begin{proof}
Projecting with center the plane $P$, we obtain a quadratic form $(\ko_{\PP^2}^3\oplus \ko_{\PP^2}(-1),q,\ko_{\PP^2}(1))$ of rank 4 over $\PP^2$ associated to the quadric bundle $\pi : \widetilde{X} \to \PP^2$,
%$$
%\left(
%\begin{array}{cccc}
%2(x - 4y - z) &   -x - 3y &   x - 3y &   2x^2 + xz - 4y^2 + 2z^2 \\
% & 2(-2y) & x - 2y - z & x^2 - xy - 3y^2 + yz - z^2\\
% &     &  2x &   2x^2 + xy + 3xz - 3y^2 + yz \\
% &     &     & 2(x^3 + x^2y + 2x^2z -xy^2 + xz^2 - y^3 + yz^2 - z^3)
%\end{array}
%\right)
%$$
see \cite[\S4.2]{hassett_varilly:K3} or
\cite[\S4]{kuznetsov:cubic_fourfolds} for the computation of the Gram
matrix.  Since $S$ is regular, $\Br(S) \to \Br(k(S))$ is injective;
see \cite{auslander_goldman} or
\cite[Cor.~1.10]{grothendieck:Brauer_II}.  By functoriality of the
Clifford algebra, the generic fiber $\beta \tensor_{S} k(S) \in
\Br(k(S))$ is represented by the even Clifford algebra of the generic
fiber $q\tensor_{\PP^2} {k(\PP^2)}$.  Thus we can perform our
calculations in the function field $k(S)$.  In the notation of
Lemma~\ref{lem:handy}, we have $m_1=2a$, $m_2=-b$, and $m_3=-2c$, and
the formulas follow immediately.
\end{proof}

\begin{remark}
Contrary to the situation in \cite{hassett_varilly:K3}, the
transcendental Brauer class $\beta \in \Br(S)$ is \emph{constant} when
evaluated on $S(\QQ)$; this suggests that arithmetic invariants do not
suffice to witness the nontriviality of $\beta$ in this case.  Indeed,
using elimination theory, we find that the odd primes $p$ of bad
reduction of $S$ are 5, 23, 263, 509, 1117, 6691, 3342589,
197362715625311, and 4027093318108984867401313726363.  For each odd
prime $p$ of bad reduction, we compute that the singular locus of
$\overline{S}_p$ consists of a single ordinary double point.  Thus by
\cite[Prop.~4.1,~Lemma~4.2]{hassett_varilly:K3_hasse}, the local
invariant map associated to $\beta$ is constant on $S(\QQ_p)$, for all
odd primes $p$ of bad reduction.  By an adaptation of
\cite[Lemma~4.4]{hassett_varilly:K3_hasse}, the local invariant map is
also constant for odd primes of good reduction.

At the real place, we prove that $S(\RR)$ is connected, hence the
local invariant map is constant.  To this end, recall that the set of
real points of a smooth hypersurface of even degree in $\PP^2(\RR)$
consists of a disjoint union of \emph{ovals} (i.e., topological
circles, each of whose complement is homeomorphic to a union of a disk
and a M\"obius band, in the language of real algebraic geometry).  In
particular, $\PP^2(\RR) \smallsetminus D(\RR)$ has a unique
nonorientable connected component $R$.  By graphing an affine chart of
$D(\RR)$, we find that the point $(1:0:0)$ is contained in $R$.  We
compute that the map projecting from $(1:0:0)$ has four real critical
values, hence $D(\RR)$ consists of two ovals.  These ovals are not
nested, as can be seen by inspecting the graph of $D(\RR)$ in an
affine chart.  The Gram matrix of the quadratic form, specialized at
$(1:0:0)$, has positive determinant, hence by local constancy, the
equation for $D$ is positive over the entire component $R$ and
negative over the interiors of the two ovals (since $D$ is smooth).
In particular, the map $f : S(\RR) \to \PP^2(\RR)$ has empty fibers
over the interiors of the two ovals and nonempty fibers over $R
\subset \PP^2(\RR)$ where it restricts to a nonsplit unramified cover
of degree 2, which must be the orientation double cover of $R$ since
$S(\RR)$ is orientable (the K\"ahler form on $S$ defines an
orientation). In particular, $S(\RR)$ is connected.

This shows that $\beta$ is constant on $S(\QQ)$.  We believe that the
local invariant map is also constant at the prime $2$, though this
must be checked with a brute force computation.
\end{remark}

%%%%%%%%%%%%%%%%%%%%%%%%%%%%%%%%%%%%%%%%%%%%%%%%%

\end{document}